\documentclass[11pt]{article}
\usepackage{amsmath,amsthm,amsfonts,amssymb,amscd, amsxtra, mathrsfs, color}
\usepackage{url}
\usepackage[margin=2cm,nohead]{geometry}

\newcommand{\banacha}{\mathbb X}
\newcommand{\banachb}{\mathbb Y}
\newcommand{\banachc}{\mathbb Z}

\newtheorem{theorem}{Theorem}
\newtheorem{lemma}[theorem]{Lemma}
\newtheorem{definition}{Definition}
\newtheorem{corollary}[theorem]{Corollary}
\newtheorem{proposition}[theorem]{Proposition}

\newtheorem{remark}{Remark}

\begin{document} 
\title{ Kantorovich's theorem on Newton's method for solving \\ strongly regular generalized equation}

\author{ 
O.  P. Ferreira\thanks{IME/UFG,  CP-131, CEP 74001-970 - Goi\^ania, GO, Brazil (Email: {\tt
      orizon@ufg.br}). The author was supported in part by  CNPq Grants   305158/2014-7, PRONEX--Optimization(FAPERJ/CNPq) and FAPEG/GO.} 
\and 
G. N. Silva \thanks{CCET/UFOB,  CEP 47808-021 - Barreiras, BA, Brazil (Email: {\tt  gilson.silva@ufob.edu.br}). The author was supported in part by CAPES .}    
}
\date{April 05, 2016}  
\maketitle
\begin{abstract}
In this paper we consider the Newton's method for solving the generalized equation of the form
$
  f(x) +F(x) \ni 0, 
$
where $f:{\Omega}\to Y$ is a   continuously   differentiable mapping, $X$  and $Y$ are Banach spaces, $\Omega\subseteq X$ an open set and $F:X \rightrightarrows  Y$ be a set-valued mapping with  nonempty  closed graph.   We show that, under strong regularity of the generalized equation,  concept introduced by S.~M.~Robinson in  \cite{Robinson1980}, and starting point satisfying the  Kantorovich's assumptions,  the Newton's method is quadratically convergent to a solution, which is unique in a suitable neighborhood  of the starting point.  The analysis presented   based on Banach Perturbation Lemma for generalized equation and the majorant technique, allow  to  unify some results  pertaining the Newton's method theory.\\

 Newton's method. majorant condition. semi-local convergence.

\end{abstract}
\section{Introduction}\label{sec:int}
In this paper we consider the Newton's method for solving the generalized equation of the form
\begin{equation} \label{eq:ipi}
  f(x) +F(x) \ni 0, 
\end{equation}
where $f:{\Omega}\to \banachb$ is a continuously differentiable mapping,  $\banacha$ and $\banachb$ are Banach spaces,  $ \Omega\subseteq \banacha$ is an open set and $F:\banacha \rightrightarrows  \banachb$ be a set-valued mapping with closed nonempty graph. As is well known, the generalized equation \eqref{eq:ipi} is an abstract model for various problems in classical analysis and its applications.  For instance, if  $\banacha=\mathbb{R}^{n}$, $\banachb=\mathbb{R}^{p+q}$ and $F=\mathbb{R}^{p}_{-}\times \{0\}$ is the product of the nonpositive orthant in $\mathbb{R}^{p}$ with the origin in  $\mathbb{R}^{q}$, then the inclusion \eqref{eq:ipi} describes a system of equalities and inequalities. If $F$ is the normal cone mapping $N_C$ of a convex set $C$ in $\banachb$, then  the inclusion \eqref{eq:ipi} is the variational inequality problem, which covers wide range of problems in mathematical programming.  Additional comments about generalized equations can be found  in \cite{DontchevAragon2014, DontchevAragon2011, Dontchev1996, DontchevRockafellar2009, DontchevRockafellar2010, DontchevRockafellar2013, Ferreira2015, josephy1979, Robinson1972_2} and the references cited therein. 

Newton's method  to solve \eqref{eq:ipi} formally generates a sequence,  for  an  initial point $x_0$,  as follows
\begin{equation} \label{eq:ipi1}
  f(x_k) + f'(x_k)(x_{k+1}-x_k)+ F(x_{k+1}) \ni 0, \qquad k=0,1, \ldots.
\end{equation}
This method may be viewed as a Newton-type method based on a partial linearization, which has been studied in several papers  including \cite{DontchevAragon2014, DontchevAragon2011, Dontchev1996, DontchevRockafellar2013};  see also \cite[ Section 6C]{DontchevRockafellar2009}. When $F\equiv 0$, the iteration \eqref{eq:ipi1} becomes the standard Newton's method for solving the nonlinear equation $f(x)=0.$  If $\banacha=\mathbb{R}^{n}$, $\banachb=\mathbb{R}^{m}$ and $F=\mathbb{R}^{s}_{-}\times \{0\}^{m-s}$, then \eqref{eq:ipi1} is a Newton's method for solving a system of equalities and inequalities; see \cite{Daniel1973}.  Now,  if $F$ is the normal cone mapping $N_C$,  of a convex set $C$ in $\banachb$ and $\banachb=\banacha^*$,  then   \eqref{eq:ipi1} is the known version of the Newton's method for solving  variational inequality;  see \cite{DokovDontchev1998, josephy1979}. In particular, if \eqref{eq:ipi} represents the Karush-Kuhn-Tucker  optimality conditions  for a mathematical programming problem, then the procedure \eqref{eq:ipi1} describes the well-known   sequential quadratic programming method; see for example  \cite[pag. 334]{DontchevRockafellar2009}. 
 
L. ~V. Kantorovich in  \cite{Kantorovich1948}, see also \cite{KantorovichAkilov1977, Polyak2007}, was the first to prove a convergence result for  Newton's method for solving  the equation $f(x)=0$,   where $f:{\Omega}\to \banachb$ is a continuously differentiable mapping,  $\banacha$ and $\banachb$ are Banach spaces and   $ \Omega\subseteq \banacha$ is an open set. Using conditions on $x_0$ the starting point, namely, under the condition that  $f'(x_0)^{-1}$ exists and $\|f'(x_0)^{-1}f(x_0)\|$ is bounded,  L.V. Kantorovich obtained well definition of the method, quadratic  convergence  and uniqueness of solution.  The idea employed in the  proof of convergence was the technique of majorization, which consists in bound the Newton's sequence by a  scalars sequence. This technique has been used and extended for various researchers, including  \cite{cibulka2015, Ferreira2015, FerreiraSilva, FerreiraSvaiter2009, Gutierrez2000, Potra2005, Wang1999, Wang2015}.  S.~M. Robinson in \cite{Robinson1972_2}, using the idea of convex process introduced by Rockafellar \cite{Rockafellar1970}, see also  \cite{Robinson1972_1, Rockafellar1967}, established a generalization of the Kantorovich's theorem for solving the inclusion $f(x)\in C$,    where $f:{\Omega}\to \banachb$ is a continuously differentiable mapping,  $\banacha$ and $\banachb$ are Banach spaces,  $ \Omega\subseteq \banacha$ is an open set and $C \subseteq \banachb$ is a nonempty closed and convex cone.  The paper \cite{Robinson1972_2} has been extended for various authors, see for instance \cite{cibulka2015, Ferreira2015, FerreiraSilva, LiNg2012}. In his Ph.D. thesis, N. H. Josephy in \cite{josephy1979} studied  Newton's method for solving the variational inequality $f(x)+N_C\ni 0$,  where $f:{\Omega}\to \mathbb{R}^{m}$ is a continuously differentiable mapping,  $ \Omega\subseteq \mathbb{R}^{n}$ is an open set and  $N_C$ is the normal cone mapping of a convex set $C\subset \mathbb{R}^{m}$. For guarantee  the well definition of the method,  {\it strong regularity} property  on   $f(x)+N_C$,   concept introduced in the theory of generalized equations by S.M. Robinson in \cite{Robinson1980}, was used.   If $\banacha = \banachb$ and $N_C=\{ 0\},$ then strong regularity  is equivalent to    $f'(x)^{-1}$  be a continuous linear operator.  If  $\banacha=\mathbb{R}^{n}$, $\banachb=\mathbb{R}^{m}$ and $F=\mathbb{R}^{s}_{-}\times \{0\}^{m-s}$, then  strong regularity is equivalent to   Mangasarian-Fromovitz constraint qualification; see \cite[Example 4D.3]{DontchevRockafellar2009}.  An important case is when \eqref{eq:ipi} represents the  Karush-Kuhn-Tucker's  systems for  the standard nonlinear programming problem with a strict local minimizer, see  \cite{DontchevRockafellar2009} pag. 232. In this case,  the strong regularity of this system is equivalent to  the  linear independence of the gradients of the active constraints and the strong second-order sufficient optimality condition; see  \cite[Theorem 6]{DontchevRockafellar96}.

The usual hypotheses used to obtain quadratic convergence of Newton's method \eqref{eq:ipi1},   for solving equation \eqref{eq:ipi},  is  the Lipschitz continuity of  $f'$  in a neighborhood  of an initial point; see  \cite{cibulka2015, DokovDontchev1998, Dontchev1996, Ferreira2015, FerreiraSilva, FerreiraSvaiter2009, josephy1979}.  Indeed,   keeping control of the derivative is an important point in the convergence analysis of Newton's method.  On the other hand, a couple of papers have dealt with the issue of convergence analysis of the Newton's method   by relaxing the assumption of Lipschitz continuity of $f'$, see for example \cite{FerreiraSvaiter2009, Wang1999, Zabrejko1987}, actually all this conditions 
are equivalent to X. Wang's condition  introduced in \cite{Wang1999}. The advantage of working with a {\it majorant condition} relaxing the assumption of Lipschitz continuity of $f'$ rests in the fact that it allow to unify several convergece results pertaining to  Newton's method; see  \cite{FerreiraSvaiter2009, Wang1999}.  In this paper we rephrase  the majorant  condition introduced in \cite{FerreiraSvaiter2009}, in order to study the convergence properties of Newton's method \eqref{eq:ipi1}. The analysis presented provides a clear relationship between the majorant function and the function defining generalized equation \eqref{eq:ipi}. Also, it allows us to obtain the  convergence radius for  the method, bound for its convergence rates  with respect to the majorant condition and uniqueness of solution.  The convergence analysis of the Newton's method under Lipschitz's  and Smale's conditions,  are provided  as special case. Up to our knowledge, this is the first time that the Newton's method to solving  generalized equations under Smale's condition in the starting point  is analyzed.  In addition, it is worth mentioning that  the recent  approach for analyzing  semi-local convergence of Newton's method and its variants,   for solving generalized equation,  use  contraction mapping principle for set-valued mappings, see  \cite{cibulka2015, DokovDontchev1998, Dontchev1996},  while our  approach is based in the  Banach Perturbation Lemma. In this sense,  our approach is related to the techniques  used in  \cite{Dontchev2015, DokovDontchev1998, josephy1979}.

The organization of the paper is as follows. In Section~\ref{sec:int.1},  some notations and important results  used  throughout  the paper are presented. In Section \ref{lkant}, the main result is stated and  in Section~\ref{sec:PR} properties of the majorant function, the main  relationships between the majorant function and the nonlinear operator are established. In Section~\ref{eq.convergenceanalysis} the main result is proved and  the uniqueness of the solution and some applications of this result are given in Section~\ref{sec:scinmer}. Some final remarks are made in  Section~\ref{sec:fr}.

\section{Preliminaries} \label{sec:int.1}
The following notations and results are used throughout our presentation.  We beginning  with the following  elementary convex analysis result:
\begin{proposition} \label{pr:conv.aux1}
Let $I\subset \mathbb{R}$ be an interval and $\varphi:I\to \mathbb{R}$ be a convex function. If $s,t,r\in I$, $s<r$, and $s\leqslant t\leqslant r$ then $\varphi(t)-\varphi(s) \leqslant \left[\varphi(r)-\varphi(s)\right][(t-s)/(r-s)].$  Moreover,  if  $\varphi$  is continuously differentiable then $\varphi':I\to \mathbb{R}$ is increasing and,  for any $s_0\in \mathrm{int}(I)$,  there holds
$$
\varphi'(s_0):={\lim}_{s\to s_0 ^-} \; \frac{\varphi(s_0)-\varphi(s)}{s_0-s}
={\sup}_{s<s_0} \;\frac{\varphi(s_0)-\varphi(s)}{s_0-s}.
$$
\end{proposition}
Let $\banacha$,\, $\banachb$ be Banach spaces,  the {\it open} and {\it closed balls} at $x$ with radius $\delta\geq 0$ are denoted, respectively, by $ B(x,\delta) = \{ y\in X ~ : ~\|x-y\|<\delta \}$ and  $B[x,\delta] = \{ y\in X ~ : ~\|x-y\|\leqslant \delta\}.$ We denote by${\mathscr L}(\banacha,\banachb)$ the {\it space consisting of all continuous linear mappings} $A:\banacha \to \banachb$ and   the {\it  norm}  of $A$ is defined  by $  \|A\|:=\sup \; \{ \|A x\|~:  \|x\| \leqslant 1 \}.$ Let  $\Omega\subseteq \banacha$ and    $h:{\Omega}\to \banachb$  a function with  Fr\'echet derivative at  all $x\in int(\Omega)$.  The Fr\'echet derivative of $h$ at $x$ is the linear mapping $h'(x):\banacha \to \banachb$ which is continuous. We identify as the {\it graph} of the set-valued mapping $H:\banacha \rightrightarrows  \banachb$ the set $\mbox{gph}~H:= \left\{(x,y)\in \banacha \times \banachb ~: ~ y \in H(x)\right\}.$ The {\it domain} and the {\it range} of $H$ are, respectively,  the sets $ \mbox{dom}~H=\{x\in \banacha ~: ~ H(x)\neq \varnothing\} $ and $ \mbox{rge}~H=\{y\in \banachb ~: ~ y \in H(x) ~\emph{for some} ~x\}$. The {\it inverse} of $H$  is the set-valued mapping  $H^{-1}:\banachb \rightrightarrows  \banacha$ defined by $ H^{-1}(y)=\{x \in \banacha ~: ~ y \in H(x)\}$.
\begin{definition}\label{def:pplm}
Let $\banacha$,\, $\banachb$ be Banach spaces,    $\Omega$ be an open nonempty subset of $\banacha$,  $h: \Omega \to \banachb$ be   a Fr\'echet differentiable  with derivative  $h'$ and $H:\banacha \rightrightarrows  \banachb$ be a set-valued mapping. The {\it partial linearization} of the mapping   $h +H$  at   $x\in \banacha$  is   the set-valued mapping   $L_h(x, \cdot ):\banacha \rightrightarrows  \banachb$ given by 
\begin{equation} \label{eq:pplm}
L_h(x, y ):=h(x)+h'(x)(y-x)+H(y).
\end{equation}
For each $x\in \banacha$, the inverse   $L_h(x, \cdot )^{-1}:\banachb \rightrightarrows  \banacha$ of the  mapping $L_h(x, \cdot )$  at $z\in Y$ is denoted  by
\begin{equation} \label{eq:invplm}
L_h(x, z )^{-1}:=\left\{y\in X ~:~ z\in h(x)+h'(x)(y-x)+H(y)\right\}.
\end{equation}
\end{definition}
\begin{remark}
If in the above  definition we have $H\equiv {0}$, $z=0$ and $h'(x)$ invertible, then the inverse mapping   $x \mapsto L_h(x, 0 )^{-1}=x-h'(x)^{-1}h(x)$ is the well known Newton's iteration mapping  for solving the equation $h(x)=0$. 
\end{remark}
An important element in the analysis of Newton's method for solving the equation $f(x)=0,$ is the behavior of inverse $f'(x)^{-1}$ for $x$ in a neighborhood of an initial point. The analogous element for  the generalized equation \eqref{eq:ipi} is the behavior  of the set-valued mapping  $L_f(x, 0 )^{-1}$,  for $x$ in a neighborhood of an initial point. It  is worth point out that,  N. H.  Josephy  in \cite{josephy1979} was the first to consider  Newton's method for solving the generalize equation $f(x)+N_{C}(x) \ni 0$, where $C$ is the normal cone of a convex set $C\subset \mathbb{R}^n$, by defining the Newton's iteration as   $L_f(x_k, ~ 0 )^{-1}\ni x_{k+1}$ for $k=0, 1, \ldots$, which is equivalent to  \eqref{eq:ipi1},  to the particular  case  $F=N_{C}$.    N. H.  Josephy  in  \cite{josephy1979}, for analyzing  Newton's method,   employed the  important concept  of strong regularity defined by S.M. Robinson \cite{Robinson1980}, which  assuring {\it ``good  behavior"} of $L_f(x, 0 )^{-1}$ for $x$ in a suitable  neighborhood of an initial point $x_0$. Here we adopt the following definition   due to Robinson  given in \cite{Robinson1980}.
\begin{definition}\label{eq:stronmetr}
Let $\banacha$,\, $\banachb$ be Banach spaces,    $\Omega$ be an open nonempty subset of $\banacha$,  $h: \Omega \to \banachb$ be   Fr\'echet differentiable  with derivative  $h'$ and $H:\banacha \rightrightarrows  \banachb$ be a set-valued mapping.  The mapping $h+H$ is said to be strongly regular at $x$ for $y$,  when $y \in h(x) +H(x)$ and there exist  constants  $r_x>0$, $r_y>0$ and  $\lambda >0$ such that $B(x,r_x)\subset \Omega$, the mapping $ z\mapsto L_h(x, z )^{-1}\cap B(x,r_x)$ is a single-valued from the ball $B(y,r_y)$ to  $B(x,r_x)$,  which is Lipschitizian on $B(y,r_y)$  with modulus  $\lambda$, i.e., 
$$
\left\|L_h(x,u)^{-1}\cap B(x,r_x)- L_h(x,v)^{-1}\cap B(x,r_x) \right\| \leq \lambda \|u-v\|, \qquad \forall ~ u, v \in  B(y,r_y). 
$$
In this case, we refer to   $\lambda$ as the Lipschtiz constant. 
\end{definition}

Since the mapping $ z\mapsto L_f(x,z)^{-1}\cap B(x_1, r_{x_1})$ is   single-valued  from $B(0, r_{0})$ to  $ B(x_1, r_{x_1})$, for simplify the notation we are using in above definition  $w= L_f(x,0)^{-1}\cap B(x_1, r_{x_1})$ instead of $\{w\}:= L_f(x,0)^{-1}\cap B(x_1, r_{x_1}).$ {\it From now on we will use  this simplified notation. }

\begin{remark} \label{re:rxinf}
If $H(x)\equiv \{0\}$ then  the property of $h+H\equiv h$ be strongly  regular at $x$ for $y$,  reduces to $h'(x)$ has an  inverse $h'(x)^{-1}$. Moreover,  in this case, the strongly regular radii  associated to $h+H$ at $x$ for $y$ are given  by   $r_x=+\infty$ and  $r_y=+\infty$, respectively,  and  the Lipschitz constant is   $\lambda=\|h'(x)^{-1}\|$.
\end{remark}
 For a  detailed discussion about  the Definition~\ref{eq:stronmetr};  see \cite{DontchevRockafellar2009, Robinson1980}.  The next result is a type of implicit function theorem for generalized equations satisfying the strongly regular condition and its proof is an immediate consequence of  \cite[Theorem 5F.4]{DontchevRockafellar2009} on page 294; see also \cite[Theorem 2.1] {Robinson1980}.
\begin{theorem}\label{eq:Implitheor}
 Let $\banacha$,  $\banachb$ and  $\banachc$ be Banach spaces, $G:\banacha \rightrightarrows  \banachb$ be  a set-valued mapping and $g:\banachc \times\banacha \to \banachb$  be a  continuous   function,   having  partial Fr\'echet derivative  with respect the second variable  $D_{x}g$  on $\banachc\times\banacha$, which is also continuous. Let $\bar{p} \in \banachc$ and suppose  that  $\bar x$ solves the generalized equation
\begin{equation} \label{eq;geift}
g(\bar p,x)+G(x) \ni 0.
\end{equation}
Assume that the  mapping  $g(\bar{p}, .)+G$ is strongly regular at $\bar{x}$ for $0$, with associated Lipschitz constant  $\lambda$. Then, for any $\epsilon >0$  there exist $r_{\bar p}>0$ and  $r_{\bar x}>0$, which depend of $\epsilon$,    and a single-valued mapping $s:B(\bar p, r_{\bar p}) \to B(\bar{x}, r_{\bar x})$ such that  for any $p\in B(\bar p, r_{\bar p})$, $s(p)$ is the unique solution in $B(\bar x, r_{\bar p})$ of the inclusion 
$
g(p,x)+G(x) \ni 0, 
$
and $s(\bar p)=\bar x$. Moreover, there holds
$$
\|s(p')-s(p)\| \leq (\lambda + \epsilon)  \|g(p',s(p))- g(p,s(p))\|,  \qquad \forall~ p,p' \in B(\bar p, r_{\bar p}).
$$
\end{theorem}
 Indeed, the  first version of the Theorem~\ref{eq:Implitheor} was proved by S.M.Robinson; see  \cite[Theorem 2.1]{Robinson1980},   to the particular  case  $F=N_{C}$, where $C$ is the normal cone of a convex set $C\subset \banacha$ and, as an application, a version involving the normal cone of the Banach Perturbation Lemma for linear operator was obtained; see \cite[Theorem~2.4]{Robinson1980}. N. H. Josephy in \cite{josephy1979},  used  this  Banach Pertubation Lemma; see \cite[Corollary~1]{josephy1979}, for  proving that the Newton iteration 
$$
  f(x_k) + f'(x_k)(x_{k+1}-x_k)+ N_{C}(x_{k+1}) \ni 0, \qquad  k=0,1,... 
$$
where $C$ is the normal cone of a convex set $C\subset \mathbb{R}^n$, is well defined and quadratically convergent for a solution of the inclusion  $f(x)+N_{C}(x) \ni 0$. In the next lemma we apply Theorem~\ref{eq:Implitheor} to obtain a version, involving a general set-valued mapping, of the Banach Perturbation Lemma for linear operator. The proof of this result is similar to the correspondent one \cite[Corollary~1]{josephy1979}.
\begin{lemma} \label{lem:blr}
Let $\banacha, \banachb$ be Banach spaces, $a_0$ be a point of $\banachb,$ $G:\banacha \rightrightarrows  \banachb$ be a set-valued mapping and $A_0:\banacha \to \banachb$ be a bounded linear mapping. Suppose that $\bar{x}$ is a point of $\banacha$ which satisfies the generalized equation
$$
0 \in A_0 x + a_0 + G(x).
$$
Assume that the mapping $ A_0  + a_0 + G$ is   strongly  regular at $\bar{x}$ for $0$ with Lipschitz constant $\lambda$. Then,  there exist $r_{\bar x}>0$,  $r_{A_0}>0$, $r_{a_0}>0$ and $r_0>0$  such that,   for any  $A\in B(A_0, r_{A_0}) \subset {\mathscr L}(\banacha,\banachb)$ and  $a\in B(a_0, r_{a_0}) \subset \banachb$  letting   $T(A, a, \cdot): B(\bar{x}, r_{\bar{x}}) \rightrightarrows  \banachb$ be   defined as
$$
T(A, a, x):= Ax + a + G(x),
$$ 
the mapping $ y\mapsto T(A, a, y)^{-1}\cap B(\bar{x}, r_{\bar{x}})$ is a  single-valued  mapping from $B(0, r_0)\subset \banachb$ to $B(\bar{x}, r_{\bar{x}})$.  Moreover,  for each   $A\in B(A_0, r_{A_0})$ and  $a\in B(a_0, r_{a_0})$ there holds  $\lambda \|A-A_0\|<1$  and the mapping $ y\mapsto T(A, a, y)^{-1}\cap B(\bar{x}, r_{\bar{x}})$ is  also  Lipschitzian  on $B(0, r_0)$ as follows  
$$
\left\|T(A, a, y_1)^{-1}\cap B(\bar{x}, r_{\bar{x}}) - T(A, a, y_2)^{-1}\cap B(\bar{x}, r_{\bar{x}}) \right\| \leq \frac{\lambda}{1-\lambda\|A-A_0\|} \|y_1-y_2\|, \qquad \forall ~y_1, y_2 \in  B(0, r_0).
$$
\end{lemma}
\begin{proof} 
Let $\banachc={\mathscr L}(\banacha, \banachb)\times \banachb$ and  $g:\banachc \times\banacha \to \banachb$ be an  operator defined by 
$
g(A, a, x)=Ax+a.
$
The operator  $g$ is  continuous  on $\banachc \times\banacha$ and has partial Fr\'echet derivative with respect to the variable $x$ given by $D_{x}g(A, a, x)=A$.   Note that  
$$
 A_0 \bar x  + a_0 + G(\bar x)= g(A_0, a_0, x)+D_{x}g(A_0, a_0, x)(\bar x-x)+G(\bar x), \qquad \forall~ x\in\banacha,
$$
 and,   by assumption,  the mapping $  A_0   + a_0 + G $   is   strongly  regular at $\bar{x}$ for $0$ with Lipschitz constant $\lambda$. Then, we may apply  Theorem~\ref{eq:Implitheor} with $\banachc={\mathscr L}(\banacha, \banachb)\times \banachb$, $\bar{p}=(A_0, a_0)$,  $p=(A, a)$ and $g(p,x)= Ax +a,$ for concluding  that, for any    $\epsilon >0$,   there exist $r_{\bar p}>0$ and  $r_{\bar x}>0$, which depend of $\epsilon$,   and a single-valued mapping $s: B(\bar{p}, r_{\bar p}) \to B(\bar{x}, r_{\bar x})$ such that  for any $(A, a)\in B(\bar{p}, r_{\bar p})$, $s(A, a)$ is the unique solution in $B(\bar{x}, r_{\bar x})$ of the inclusion 
$$
T(A,a,x):=Ax+a+G(x) \ni 0, 
$$
and $s(A_0, a_0)=\bar x$.  Moreover, the following inequality holds
$$
\|s(A,a)-\bar{x}\|\leq (\lambda +\epsilon)\|(A-A_0)\bar{x} + (a-a_0)\|,  \qquad \forall ~(A,a)\in B(\bar{p}, r_{\bar p}).
$$
Thus, the  single-valued mapping $s$ is bounded and  we can choose  $r_{A_0}>0$, $r_{a_0}>0$ and $r_0>0$  such that $B(A_0, r_{A_0})\times [B(a_0, r_{a_0})-B(0, r_{0})]\subset B(\bar{p}, r_{\bar{p}}),$
and    for each   $A\in B(A_0, r_{A_0})$,   $a\in B(a_0, r_{a_0})$  and  $y\in B(0, r_{0})$  there holds  
$$
\lambda \|A-A_0\|<1, \qquad \qquad   y+(A_0-A)s(A, a)+(a_0-a) \in B(0, \hat r_{0}), 
$$
where the  radius $\hat r_{0}>0$  is given in the definition of strong regularity of  $  A_0   + a_0 + G $  at $\bar{x}$ for $0$. Let  $A\in B(A_0, r_{A_0})$, $a\in B(a_0, r_{a_0})$ and  $y_1,y_2\in B(0, r_{0})$, and  let $s(A, a-y_1) $ and $s(A, a-y_2) $  be the solutions associated with $y_1$ and $y_2$, respectively. Since $T(A,a, s(A, a-y_i))\ni y_i$,  i.e.,   $s(A, a-y_i)=T(A,a,y_i)^{-1}\cap B(\bar{x}, r_{\bar x})$, for $ i=1,2$,   after some manipulation,  we obtain that
\begin{equation} \label{eq:yiii}
y_i +(A_0-A)s(A, a-y_i) + (a_0 -a) \in   A_0 s(A, a-y_i)  + a_0 + G(s(A, a-y_i)), \qquad  i=1,2.
\end{equation}
Therefore, taking into account that $  A_0 + a_0 + G $ is strongly regular at $\bar{x}$ for $0$ with associated Lipschitz constant $\lambda$,  the inclusions in \eqref{eq:yiii} imply that
\begin{multline*}
\left\|s(A, a-y_1)-s(A, a-y_2)\right\| \leq \\ \lambda \left\|[y_1 +(A_0-A)s(A, a-y_1) + (a_0 -a)] -[y_2 +(A_0-A)s(A, a-y_2)+ (a_0 -a)]\right\|.
\end{multline*}
Using properties of the norm, last inequality becomes to
$$
\left\|s(A, a-y_1)-s(A, a-y_2)\right\| \leq  \lambda \|y_1-y_2\| + \lambda\|A_0-A\|\|s(A, a-y_1) -s(A, a-y_2)\|.
$$
 Now, since $\lambda\|A-A_0\|<1$ for each $A\in B(A_0, r_{A_0})$, then  last inequality implies that
$$
\left\|s(A, a-y_1)-s(A, a-y_2)\right\| \leq  \frac{\lambda}{1-\lambda\|A-A_0\|}\|y_1-y_2\|, 
$$
and the result follows by notting that  $s(A, a-y)=T(A,a,y)^{-1}\cap U$ and $y_1, y_2\in  B(0, r_{0})$ are arbitrary.
\end{proof}

Next we  establish a corollary to Lemma~\ref{lem:blr}, which  will have important rule in the sequel. 
\begin{corollary} \label{cor:ban}
Let $\banacha$,\, $\banachb$ be Banach spaces,    $\Omega$ be an open nonempty subset of $\banacha$,  $f: \Omega \to \banachb$ be continuous with   Fr\'echet derivative  $f'$ continuous, and $F:\banacha \rightrightarrows  \banachb$ be a set-valued mapping. Suppose that $x_0 \in \Omega$ and the mapping $ L_f(x_0,  . ):\banacha \rightrightarrows  \banachb$ is strongly regular at $x_1$ for $0$ with associated  Lipschitz  constant $\lambda >0$.   Then, there exist three  constants $r_{x_1}$,  $r_{0}>0$ and $ r_{x_0}>0$ such that, for each  $x\in B(x_0, r_{x_0})$,  there holds  $\lambda \|f'(x)-f'(x_0)\|<1$, the mapping $ z\mapsto L_f(x,z)^{-1}\cap B(x_1, r_{x_1})$ is single-valued  from $B(0, r_{0})$ to  $B(x_1, r_{x_1})$ and Lipschitizian   as follows  
$$
\left\|L_f(x,u)^{-1}\cap B(x_1, r_{x_1})- L_f(x,v)^{-1}\cap B(x_1, r_{x_1}) \right\| \leq \frac{\lambda}{1-\lambda\|f'(x)-f'(x_0)\|} \|u-v\|, \quad \forall ~ u, v \in  B(0, r_{0}).
$$
\end{corollary}
\begin{proof}
Since   $ L_f(x_0,  . ):\banacha \rightrightarrows  \banachb$ is strongly regular at $x_1$ for $0$ with associated  Lipschitz  constant $\lambda >0$, applying first part of Lemma~\ref{lem:blr} with   $\bar{x}=x_1$,  $A_0=f'(x_0),$ $a_0= f(x_0)-f'(x_0)x_0$ and  $G=F, $  we conclude  that there exist $r_{x_1}>0$,  $\tilde r>0$, $\hat r>0$ and $r_0>0$  such that,   for any  $A\in B(f'(x_0), \tilde r) \subset {\mathscr L}(\banacha,\banachb)$ and  $a\in B( f(x_0)-f'(x_0)x_0, \hat r) \subset \banachb$ ,  letting   $T(A, a, \cdot): B(x_1, r_{x_1}) \rightrightarrows  \banachb$ be   defined as
$$
T(A, a, y):= Ay + a + F(y), 
$$
the mapping $ z\mapsto T(A, a, z)^{-1}\cap B(x_1, r_{x_1})$ is a  single-valued  mapping from $B(0, r_{0})$ to $B(x_1, r_{x_1})$.  Due to  $f$ be continuous with    $f'$ continuous, then  there exists $r_{x_0}>0$ such that  $\lambda \|f'(x)-f'(x_0)\|<1$, 
$$
f'(x ) \in  B(f'(x_0), \tilde r), \qquad    f(x)-f'(x)x \in B( f(x_0)-f'(x_0)x_0, \hat r), \qquad    \qquad \forall ~x \in B(x_0, r_{x_0}).
$$
Hence,  we conclude that for each $ x \in B(x_0,  r_{x_0})$,  the mapping   $ z\mapsto T(f'(x ),  f(x)-f'(x)x, z)^{-1}\cap B(x_1, r_{x_1})$ is a  single-valued  from $B(0, r_{0})$ to $B(x_1, r_{x_1})$, where

\begin{equation}\label{eq;tgco}
T(f'(x ),  f(x)-f'(x)x, y):= f'(x)y + f(x)-f'(x)x + F(y)=f(x) +f'(x)(y-x)+F(y).
\end{equation}
Since Definition~\ref{def:pplm} and \eqref{eq;tgco} imply that   $L_f(x, y)=T(f'(x), f(x)-f'(x)x, y)$, for all $x\in B(x_0, r_{x_0})$ and $y\in B(x_1, r_{x_1}),$ after some manipulations we have,  for each $z\in B(0, r_{0})$,
\begin{equation} \label{eq:elft}
L_f(x, z)^{-1}\cap B(x_1, r_{x_1})=T(f'(x), f(x)-f'(x)x, z)^{-1}\cap B(x_1, r_{x_1}), \qquad \forall ~x\in B(x_0, r_{x_0}). 
\end{equation}
Therefore, for each $x\in B(x_0, r_{x_0})$,  the last equality and \eqref{eq;tgco} imply that   $ z\mapsto L_f(x,z)^{-1}\cap B(x_1, r_{x_1})$ is   single-valued  from $B(0, r_{0})$ to  $B(x_1, r_{x_1})$, which  proof the first part of corollary. Finally, taking into account  \eqref{eq:elft} and  second part of Lemma~\ref{lem:blr},  we also conclude that the mapping   $ z\mapsto L_f(x,z)^{-1}\cap B(x_1, r_{x_1})$  is Lipschitzian from $B(0, r_{0})$ to $B(x_1, r_{x_1})$ with Lipschitz constant $\lambda/[1-\lambda\|f'(x)-f'(x_0)\|], $  which conclude the proof.
\end{proof}

\begin{remark}
If in above corollary we have $F\equiv 0$. Then, for each $x\in B(x_0, r_{x_0})$, the mapping   $ z\mapsto L_f(x,z)^{-1}\cap B(x_1, r_{x_1})$ be    single-valued   from $B(0, r_{0})$ to  $B(x_1, r_{x_1})$  means that $f'(x)$ is inventible and,   for each  $z\in B(0, r_{0})$,  there exists a unique  $y\in B(x_1, r_{x_1})$ such that  $y=L_f(x,z)^{-1}\cap B(x_1, r_{x_1})=x+f'(x)^{-1}(z-f(x))$. Moreover,   $\lambda \|f'(x)-f'(x_0)\|<1$ and there holds
$$
\left\|f'(x)^{-1}(u-v)\right\| = \left\|[x+f'(x)^{-1}(u-f(x))]-[x-f'(x)^{-1}(v-f(x))]\right\|\leq\frac{\lambda\|u-v\|}{1-\lambda \|f'(x)-f'(x_0)\|}, 
$$
for all $u, v\in \banachb$ and $x\in B(x_0, r_{x_0})$. Therefore,  from Remark~\ref{re:rxinf},  we have  $r_{x_1}=r_{0}=+\infty$ and $\lambda=\|f'(x_0)^{-1}\|$ and, last inequality becomes

$$
\|f'(x)^{-1}\|\leq \frac{\|f'(x_0)^{-1}\|}{1-\|f'(x_0)^{-1}\| \|f'(x)-f'(x_0)\|}, \qquad  \forall  ~x\in B(x_0, r_{x_0}).
$$
\end{remark}
\section{ Kantorovich's theorem for Newton's method} \label{lkant}
In this section, our goal is to state and prove a Kantorovich's theorem for Newton's method  for solving the generalized equation of the form \eqref{eq:ipi}.  To state the theorem we need to fix  some important  constants.  Let $\banacha$,\, $\banachb$ be Banach spaces,    $\Omega$ be an open nonempty subset of $\banacha$,  $f: \Omega \to \banachb$ be continuous with   Fr\'echet derivative  $f'$ continuous, and $F:\banacha \rightrightarrows  \banachb$ be a set-valued mapping. {\it From now on,  for  $x_0 \in \Omega$ and  a partial linearization mapping $ L_f(x_0,  . ):\banacha \rightrightarrows  \banachb$   at   $x_0$,  given by   
$$
L_f(x_0,  x ):=f(x_0)+f'(x_0)(x-x_0) + F(x), 
$$
strongly regular at $x_1$ for $0$ with associated Lipschitz constant $\lambda$,   we refer to the real  numbers
\begin{equation} \label{eq:srcm}
r_{x_1}>0,  \qquad \qquad r_{0}>0,    \qquad \qquad r_{x_0}>0, 
\end{equation}
as   the three  constants given by Corollary~\ref{cor:ban}}. The statement of main result is:
\begin{theorem}\label{th:nt}
Let $\banacha$, $\banachb$ be Banach spaces, $\Omega\subseteq \banacha$ an open set,   $f: \Omega \to \banachb$ be continuous with   Fr\'echet derivative  $f'$ continuous  and $F:\banacha \rightrightarrows  \banachb$ be a set-valued mapping with closed graph. 
 Let $x_0\in \Omega$, $R>0$  and  $\kappa:=\sup\{t\in [0, R): B(x_0, t)\subset \Omega\}$. Suppose that the  partial linearization mapping $ L_f(x_0,  . ):\banacha \rightrightarrows  \banachb$   at   $x_0$,  is strongly regular at $x_1\in \Omega$ for $0$ with  associated  Lipschitz constant $\lambda >0$  and  there exist $\psi:[0,\; R)\to \mathbb{R}$ twice continuously differentiable  function  such that  
  \begin{equation}\label{Hyp:MH}
\lambda\left\|f'(y)-f'(x)\right\| \leq     \psi'(\|y-x\|+ \|x-x_0\|)-\psi'(\|x-x_0\|),
  \end{equation}
  for all $x, y \in B(x_0, \kappa)$ and $\|y-x\|+ \|x-x_0\|<R$. Moreover, suppose that
	\begin{equation} \label{eq.ipoint}
	\|x_1-x_0\|\leq \psi(0),
	\end{equation}
	and the following conditions hold:
\begin{itemize}
  \item[{\bf h1)}]  $\psi(0)>0$,   $\psi'(0)=-1$;
  \item[{\bf h2)}]  $\psi'$ is convex and strictly increasing;
  \item[{\bf h3)}]  $\psi(t)=0$ for some $t\in (0,R)$ and let $t_*:=\min \{t\in [0, R)  ~:~  \psi(t)=0 \}$.
  \end{itemize}
 Additionally, for the constants  $r_{0}$ and $r_{x_0}$   fixed  in  \eqref{eq:srcm}, suppose that    the following  inequalities  hold:
\begin{equation}\label{eq:conddef}
 t_* \leq r_{x_0}, \qquad \qquad \qquad \frac{\psi''(t_*)}{2\lambda}\psi(0)^2 < r_0.
\end{equation}
Then,  the sequences generated by Newton's method for solving the generalized equation $0\in f(x) + F(x)$ and the equation $\psi(t)=0,$ with starting point $x_0$ and $t_0=0,$ defined respectively by,
\begin{equation} \label{eq:DNS}
x_{k+1}:=L_{f}(x_k, 0)^{-1} \cap B(x_1, r_{x_1}), \qquad t_{k+1} ={t_k}-\psi(t_k)/\psi'(t_k),\qquad k=0,1,\ldots\,, 
\end{equation}
are well defined, $\{t_k\}$ is strictly increasing, is contained in $(0, t_ *)$ and converges to $t_ *$, $\{x_k\}$ is contained in $B(x_0,t_ *)$ and converges to the point $x_*\in B[x_0, t_ *]$, which is the unique solution of the generalized equation $0\in f(x)+F(x)$ in $B[x_0, t_ *]\cap B[x_{1}, r_{x_1}]$. Moreover, $\{x_k\}$ and $\{t_k\}$ satisfies

\begin{equation}\label{eq:q2}
   \|x_*-x_k\| \leq t_* -t_k, \qquad \qquad   \|x_*-x_{k+1}\| \leq \frac{t_*-t_{k+1}}{(t_* -t_k)^2}\|x_*-x_k\|^2,  
  \end{equation}
for all k=0,1,..., and the sequences  $\{x_k\}$ and $\{t_k\}$ converge $Q$-linearly as follows
\begin{equation}\label{eq:rates0}
\|x_*-x_{k+1}\| \leq \frac{1}{2}\|x_* -x_k\|, \qquad \qquad  t_* -t_{k+1} \leq \frac{1}{2} (t_* -t_k),  \qquad k=0,1, \ldots\ .
\end{equation}
Additionally, if  the following condition holds
\begin{itemize}
  \item[{\bf h4)}] $\psi'(t_*)<0$,
	\end{itemize}
	then the sequences,  $\{x_k\}$  and $\{t_k\}$ converge $Q$-quadratically as follows  
	\begin{equation}\label{ine.rates1}
	\|x_*-x_{k+1}\| \leq \frac{\psi''(t_*)}{-2\psi'(t_*)}\|x_*-x_k\|^2, \qquad \qquad  t_* -t_{k+1} \leq \frac{\psi''(t_*)}{-2\psi'(t_*)} (t_* -t_k)^2, \qquad k=0,1, \ldots\ .
	\end{equation}
\end{theorem}
 \begin{remark}
In  Section~\ref{sec:scinmer}, we will present  several particular instances of Theorem~\ref{th:nt},  by presenting the explicit majorant function. For instance, when     $F\equiv \{0\}$   and    $f'$ satisfies a Lipschitz-type condition, i.e., the majorant function associated to $f'$ is a quadratic  polinomial defined by the Lipschitz constant, we retrieve  a version of the  classical Kantorovich's  theorem on Newton's method;  for example, see \cite{{Kantorovich1948}, KantorovichAkilov1977}.  
\end{remark}

Henceforward we assume that all  the assumptions in  Theorem~\ref{th:nt} holds.
\subsection{Basic results} \label{sec:PR}
In this section we will establish some results about the majorant function $\psi:[0,\; R)\to \mathbb{R}$ and, some relationships between the majorant function and the set-valued mapping $f+F.$ We begin by reminding that Proposition~3 of \cite{FerreiraSvaiter2009} state that the majorant function $\psi$  has a smallest root $t_*\in  (0,R)$, is strictly convex, $\psi(t)>0$ and  $\psi'(t)<0$,  for all $t\in [0,t_ *)$. Moreover, $\psi'(t_*)\leqslant 0$ and $\psi'(t_*)<0$ if, and only if, there exists $t\in (t_*,R)$ such that $\psi(t)< 0$.   Since $\psi'(t)<0$ for all $ t\in [0, t_*)$, the Newton iteration of the majorant function $\psi$ is well defined in $[0, t_*).$ Let us call it $n_{\psi}: [0, t_*) \to \mathbb{R}$ such that 
\begin{equation}\label{eq.majorfunc}
n_{\psi}(t)=t-\frac{\psi(t)}{\psi'(t)}.
\end{equation}
The next result will be used to obtain the convergence rate of the sequence generated by Newton's method for solving $\psi(t)=0.$ Its proof can be found in \cite[Proposition 4]{FerreiraSvaiter2009}.
\begin{lemma}\label{eq.ratemajor}
For all $t\in [0,t_*)$ we have $n_{\psi}(t)\in [0,t_*),$ $t<n_{\psi}(t)$ and $t_*-n_{\psi}(t)\leq \frac{1}{2}(t_*-t).$ Moreover, the Newton step  function   $ [0, t_*) \mapsto   -\psi(t)/\psi'(t) \in [0, +\infty) $  is decreasing. If $\psi$ also satisfies {\bf h4}  then
$$
t_* -n_{\psi}(t) \leq \frac{D^{-}\psi'(t_*)}{-2\psi'(t_*)} (t_* -t)^2, \qquad \forall ~t\in [0,t_*).
$$
\end{lemma}
Using \eqref{eq.majorfunc}, the definition of $\{t_k\}$ in \eqref{eq:DNS} is equivalent to the following one
\begin{equation}\label{eq.majseq}
t_0=0, \qquad t_{k+1}=n_{\psi}(t_k), \qquad k=0,1\ldots.
\end{equation}
The next result contain the main convergence properties of the above sequence and its prove, which is a consequence of Lemma~\ref{eq.ratemajor}, follows the same pattern as the proof of Corollary~2.15 of \cite{FerreiraMax2013}.
\begin{corollary}\label{major.convergence}
The sequence $\{t_k\}$ is well defined, strictly increasing and is contained in $[0,t_*).$ Moreover,  $\{t_k\}$  converges $Q$-linearly to $t_*$  as  the  second inequality in \eqref{eq:rates0}.  Additionally, If  {\bf h4} holds,   then $\{t_k\}$ converges $Q$-quadratically  to $t_*$ as the second inequality in \eqref{ine.rates1} and converges $Q$-quadratically.
\end{corollary}
Therefore, we have obtained all the statements about the majorant sequence $\{t_k\}$ in Theorem~\ref{th:nt}.  Now,  we are going  to  establish  some relationships between the majorant function and the set-valued mapping $f+F.$ The next result is a consequence of Corollary \ref{cor:ban}.
\begin{proposition} \label{wdns}
For any $x\in B(x_0, t_*)$, the mapping $ z\mapsto L_f(x,z)^{-1}\cap B(x_1, r_{x_1})$ is single-valued from $B(0, r_{0})$ to $ B(x_1, r_{x_1})$ and there holds
$$
\left\|L_f(x,u)^{-1}\cap B(x_1, r_{x_1})- L_f(x,v)^{-1}\cap B(x_1, r_{x_1}) \right\| \leq  -\frac{\lambda}{\psi'(\| x-x_0\|)} \|u-v\|, \qquad \forall ~ u, v \in  B(0, r_{0}).
$$
\end{proposition}
\begin{proof}
Definitions of the constants   $r_{x_1}$, $r_{0}$  and $r_{x_0}$    in  \eqref{eq:srcm}  together with   Corollary~\ref{cor:ban} imply that, for any $x\in B(x_0, r_{x_0})$,   the mapping $ z\mapsto L_f(x,z)^{-1}\cap B(x_1, r_{x_1})$ is   single-valued  from $B(0, r_{0})$ to  $B(x_1, r_{x_1})$ and  Lipschitizian   as follows  
 \begin{equation}\label{eq:majcond}
\left\|L_f(x,u)^{-1}\cap B(x_1, r_{x_1})- L_f(x,v)^{-1}\cap B(x_1, r_{x_1}) \right\| \leq \frac{\lambda}{1-\lambda\|f'(x)-f'(x_0)\|} \|u-v\|, 
\end{equation}
for all $ u, v \in  B(0, r_{0})$. Since $\| x-x_0\|<t_*$ thus $\psi'(\|x-x_0\|)<0$. Hence,    \eqref{Hyp:MH}  together  with  {\bf h1} imply that
$$
     \lambda\|f'(x)-f'(x_0)\| \leq \psi'(\|x-x_0\|)-\psi'(0)<1, \qquad \forall ~ x\in B(x_0, t_*).
$$
Using assumption in \eqref{eq:conddef}, i.e.,  $ t_*   \leq r_{x_0}$, last  inequality,   \eqref{eq:majcond} and {\bf h1},  we concluded that the inequality of the  proposition  holds,  for all $x\in B(x_0, t_*)$. 
\end{proof}
Newton's iteration at a point of a neighborhood of $x_0$ happens to be a zero of the partial  linearization
of $f+F$ at such a point.  Therefore, we first  study the  linearization error of $f$ at  points in $\Omega$
\begin{equation}\label{eq:def.er}
  E_f(x,y):= f(y)-\left[ f(x)+f'(x)(y-x)\right],\qquad \forall ~y,\, x\in \Omega.
\end{equation}
In the next result we will bound this error by the linearization error of the majorant function $\psi$, namely,
\begin{equation}\label{eq:def.erf}
        e_{\psi}(t,u):= \psi(u)-\left[ \psi(t)+\psi'(t)(u-t)\right],\qquad  \forall ~ t,\,u \in [0,R).
\end{equation}
\begin{lemma} \label{pr:taylor}
  Take $x,y\in B(x_0,R)$ and $0\leq t<v< R$. If $\|x-x_0\|\leq t$ and $\|y - x\|\leq v-t$ then
\begin{equation}\label{eq:errormajor}
	\lambda \|E_f(x,y)\| \leq e_{\psi}(t,v)\frac{\|y-x\|^2}{(v-t)^2}\leq \frac{1}{2}\psi''(v)(v-t)^2.
\end{equation}
\end{lemma}
\begin{proof}
Since $x+\tau(y-x)\in B(x_0,R),$ for all $\tau\in [0,1]$ and $f$ is continuously differentiable in $\Omega$, the linearization error of $f$ in \eqref{eq:def.er} is equivalent to 
$$
 E_f(x,y)=\int_{0}^{1} [f'(x+\tau(y-x))-f'(x)](y-x) d\tau, 
$$
which combined with the  assumption in \eqref{Hyp:MH} and after some  simple algebraic manipulations we obtain 
\begin{equation} \label{eq:fiftc}
\lambda \|E_f(x,y)\| \leq \int_{0}^{1} [\psi'(\|x-x_0\| +\tau \|y-x\|) -\psi'(\|x-x_0\|)]\|y-x\| d\tau.
\end{equation}
Using assumption {\bf h2}, we know that $\psi'$ is convex. Thus, since   $\|x-x_0\|\leq t$ we conclude that 
$$
\psi'(\|x-x_0\| +\tau \|y-x\|) -\psi'(\|x-x_0\|)  \leq \psi'(t +\tau \|y-x\|) -\psi'(t),  \qquad   \forall ~\tau \in [0,1].
$$
Due to $\|y-x\|<v-t$  and $v<R$, first statement in  Proposition~\ref{pr:conv.aux1} together  with last inequality implies 
$$
   \psi'(\|x-x_0\| +\tau \|y-x\|) -\psi'(\|x-x_0\|) \leq [\psi'(t +\tau \|v-t\|) -\psi'(t)]\frac{\|y-x\|}{v-t},  \qquad   \forall ~ \tau \in [0,1].
$$
 Combining the inequality in \eqref{eq:fiftc}  with last  inequality we conclude that
	$$
	\lambda \|E_f(x,y)\| \leq \int_{0}^{1} [\psi'(t +\tau \|v-t\|) -\psi'(t)]\frac{\|y-x\|^2}{v-t}d\tau, 
	$$
	which, after performing the integration yields \eqref{eq:errormajor}.   Now,  we are going to prove the last inequality in \eqref{eq:errormajor}. Definition in \eqref{eq:def.erf}  implies
	$$
	e_{\psi}(t,  v) =\int_0^1 [\psi'(t +\tau (t-v)) -\psi'(t)](v-t)d\tau.
	$$
	We know that  $\psi'$ is convex. Thus, using  the  first and second statement in  Proposition~\ref{pr:conv.aux1},  it follows from last equality that 
	$$
	e_{\psi}(t,  v)  \leq \int_0^1\frac{\psi'(v)-\psi'(t)}{v-t} \tau (v-t)^2d\tau   \leq \int_0^1\psi''(v) \tau (v-t)^2d\tau=\frac{1}{2}\psi''(v)(v-t)^2, 
	$$
	which, using first inequality in \eqref{eq:errormajor} and  considering that $\|y-x\|\leq v-t$,   gives the desired inequality.
\end{proof}
 Proposition ~\ref{wdns} guarantees, in particular,  that  for each  $x\in B(x_0, t_*)$ the mapping $ z\mapsto L_f(x,z)^{-1}\cap B(x_1, r_{x_1})$ is   single-valued  from $B(0, r_{0})$ to  $ B(x_1, r_{x_1})$  and consequently, the Newton iteration mapping  is well-defined.  Let us call $N_{f+F}$, the Newton
iteration mapping  for $f+F$ in that region, namely, $N_{f+F}:B(x_0, t_*) \to \banacha$ is defined by 
\begin{equation} \label{NF}
N_{f+F}(x):= L_f(x,0)^{-1}\cap B(x_1, r_{x_1}).
\end{equation}
Using \eqref{eq:invplm} we conclude that the definition of the Newton iteration mapping  in \eqref{NF} is equivalent to 
\begin{equation} \label{eq:NFef}
0\in f(x)+f'(x)(N_{f+F}(x)-x)+F(N_{f+F}(x)),\qquad  N_{f+F}(x)\in B(x_1, r_{x_1}),   \qquad \forall ~x\in  B(x_0, t_*).
\end{equation}
Therefore, one can apply a \emph{single} Newton iteration on any $x\in B(x_0, t_*)$ to obtain $N_{f+F}(x)$ which may not belong
to $B(x_0,  t_*)$. Thus, this is enough to guarantee the  well-definedness of only one iteration.   To ensure that Newtonian iterations may be repeated indefinitely or,  in particular, invariant on subsets of    $B(x_0, t_*)$,  we need some additional results. First, define some subsets of $B(x_0,t_*)$ in which, as we shall prove, Newton iteration mapping  \eqref{NF} are ``well behaved''. Define
\begin{equation}\label{eq:ker} 
{K}(t):=\left\{x\in \Omega ~: ~ \|x-x_0\| \leq t, \quad  \|L_f(x,0)^{-1}\cap B(x_1, r_{x_1})-x\| \leq -\frac{\psi(t)}{\psi'(t)}\right\}, \qquad t\in [0,t_*), 
\end{equation} 
\begin{equation} \label{eq:kt}  
{K}:=\bigcup_{t\in {[0,t_ *)}}
  K(t).
\end{equation}
\begin{proposition} \label{pr:syn} 
 For each $0\leq t< t_*$ we have ${K}(t) \subset B(x_0,t_*)$ and 
	$
	N_{f+F}({K}(t)) \subset {K}(n_{\psi}(t)).
	$
	As a consequence, ${K}\subseteq B(x_0,t_ *)$ and $N_{f+F}({K}) \subset {K}$.
\end{proposition} 
\begin{proof}
The first inclusion follows trivially from the definition of ${K}(t).$ Take   $x\in {K}(t)$ and,  from definitions \eqref{eq:ker} and \eqref{eq.majorfunc}, follow that
\begin{equation}\label{ini.cond}
\|x-x_0\| \leq t, \quad \qquad \|L_f(x,0)^{-1}\cap B(x_1, r_{x_1})-x\| \leq -\frac{\psi(t)}{\psi'(t)}, \qquad \quad  t< n_{\psi}(t) <t_ *.
\end{equation}
Definition of  Newton iteration mapping in \eqref{NF} implies that,   for all $x  \in {K}(t)$ there holds
$$
\|N_{f+F}(x)-x_0\| \leq \|x-x_0\| + \|N_{f+F}(x)-x\|= \|x-x_0\| + \|L_f(x,0)^{-1}\cap B(x_1, r_{x_1})-x\|, 
$$
and consequently,  using  \eqref{eq.majorfunc} and \eqref{ini.cond}, the last inequality imply  that 
\begin{equation} \label{eq:fcmt}
\|N_{f+F}(x)-x_0\| \leq t -\frac{\psi(t)}{\psi'(t)} = n_{\psi}(t) <t_ *.
\end{equation}
For simplify the notations, let $x_+= N_{f+F} (x) \in B(x_1, r_{x_1})$.  Thus, using    \eqref{eq:NFef} and definition in  \eqref{eq:pplm}  we have
$$
0\in  L_f (x,x_+)= f(x)+f'(x)(x_+ -x)+F(x_+).
$$
After some simple manipulations in last inequality and taking into account \eqref{eq:def.er} we obtain that
\begin{eqnarray*}
0&\in& -f(x_+) + f(x)+f'(x)(x_+ -x)+  f(x_+)+f'(x_+)(x_+ -x_+)+F(x_+)\\
&=& -E_{f}(x,x_+)+ f(x_+)+f'(x_+)(x_+ -x_+) +F(x_+).
\end{eqnarray*}
Using    \eqref{eq:pplm}, we conclude that the last inclusion is equivalent to  $E_{f}(x,x_+) \in L_{f}(x_+, x_+), $ which implies that
\begin{equation} \label{eq:xp}
x_+ \in L_f(x_+,E_{f}(x,x_+))^{-1}\cap B(x_1, r_{x_1}).
\end{equation}
Since  the majorant function $\psi$  has a smallest root $t_*\in (0,R) $, we have from \eqref{eq:fcmt} that $x_+ \in B[x_0, t_*]$.  Now,  we are going to prove that 
\begin{equation}  \label{eq:ebbro}
E_{f}(x,x_+) \in B[0,r_0].
\end{equation}
Since $x\in K(t)$, definitions \eqref{eq.majorfunc} and   \eqref{NF} together  with \eqref{ini.cond}  imply  that  $ t< n_{\psi}(t)$ and $\|x_+-x\|\leq  n_{\psi}(t) -t$. Thus, applying second inequality in Lemma~\ref{pr:taylor} with $y=x_+$ and $v=n_{\psi}(t)$ we conclude that
$$
\lambda \|E_f(x,x_+)\| \leq  \frac{1}{2}\psi''(n_{\psi}(t))(n_{\psi}(t)-t)^2.
$$
On the other hand, from   {\bf h2} we have $\psi''$ is increasing and Lemma~\ref{eq.ratemajor} together  {\bf h1} gives  $n_{\psi}(t) -t=-\psi(t)/\psi'(t)\leq -\psi(0)/\psi'(0) =\psi(0)$.  Thus, above inequality becomes
$$
\lambda \|E_f(x,x_+)\| \leq  \frac{1}{2}\psi''(t_*)\psi(0)^2.
$$ 
Therefore, using \eqref{eq:conddef}    we obtain  the desired inclusion in \eqref{eq:ebbro}.  Hence, since $x_+ \in B[x_0, t_*]$,  combining \eqref{eq:xp} with \eqref{eq:ebbro} and  first part of Proposition~\ref{wdns}, we conclude that $x_+=L_f(x_+,E_{f}(x,x_+))^{-1}\cap B(x_1, r_{x_1})$. Thus,  using the second part of Proposition~\ref{wdns}  we have
$$
\|L_f(x_+,0)^{-1}\cap B(x_1, r_{x_1})-x_+\| \leq -\frac{\lambda}{\psi'(\|x_+-x_0\|)}\|E_{f}(x,x_+)\|.
$$
Due to  $x_+= N_{f+F} (x)$ we have from \eqref{eq:fcmt}  that $\|x_+-x_0\| \leq  n_{\psi}(t)$. Then, taking into account that  $\psi'$ is increasing and negative, it follows from  above inequality,  Lemma~\ref{pr:taylor}, \eqref{NF} and \eqref{ini.cond}  that 
$$
\|L_f(x_+,0)^{-1}\cap B(x_1, r_{x_1})-x_+\| \leq -\frac{\lambda}{\psi'(n_{\psi}(t))}\|E_{f}(x,x_+)\|
											\leq -\frac{e_{\psi}(t,n_{\psi}(t))}{\psi'(n_{\psi}(t))}\frac{\|x_+-x\|^2}{(n_{\psi}(t)-t)^2}.
$$
On the other hand, using   the definition  \eqref{eq.majorfunc} and \eqref{eq:def.erf}, after  some manipulations we conclude that
$$
\psi(n_{\psi}(t))= \psi(n_{\psi}(t)) -[\psi(t) + \psi'(t)(n_{\psi}(t)-t)]= e_{\psi}(t,n_{\psi}(t)), 
$$
 and because  $x_+= N_{f+F} (x)$,   \eqref{eq.majorfunc} and the  second inequality in \eqref{ini.cond}  imply  $\|x-x_+\|\leq n_{\psi}(t)-t$, above inequality becomes
$$
\|L_f(x_+,0)^{-1}\cap B(x_1, r_{x_1})-x_+\| \leq -\frac{\psi(n_{\psi}(t))}{\psi'(n_{\psi}(t))}.
$$
Therefore, since \eqref{eq:fcmt}  implies $\|x_+-x_0\| \leq n_{\psi}(t)$  we conclude that the second inclusion of the proposition  is proved.

 The third  inclusion ${K}\subseteq B(x_0,t_*)$ follows trivially from \eqref{eq:ker} and \eqref{eq:kt}. To prove the last inclusion $N_{f+F}({K}) \subset {K}$, take $x\in {K}$. Thus,  $x\in K(t)$ for some $t\in [0,t_ *)$. From the second inclusion  of the proposition, we have $N_{f+F}(x) \in {K}(n_{\psi}(t))$. Since $n_{\psi}(t)\in [0,t_ *)$ and using the definition of ${K}$ in \eqref{eq:kt} we conclude the proof.
\end{proof}

\subsection{Convergence analysis}\label{eq.convergenceanalysis}
To prove the convergence result, which is a consequence of the above results, firstly we note that the definition \eqref{NF} implies that the sequence $\{x_k\}$ defined in \eqref{eq:DNS},  can be  formally stated by 
\begin{equation}\label{eq.seq}
x_{k+1}=N_{f+F}(x_k), \qquad k=0,1,\ldots,
\end{equation}
or equivalently, 
\begin{equation} \label{eq:DNSa}
 0\in f(x_k)+f'(x_k)(x_{k+1}-x_k)+F(x_{k+1}),  \qquad  x_{k+1} \in B(x_1,r_{x_1}), \qquad k=0,1, \ldots.
\end{equation}
First we will  show that the sequence generated by Newton method converges to  $x_*\in B[x_0,t_ *]$,   a solution of the generalized equation \eqref{eq:ipi},  and  is well behaved with respect to the set defined in \eqref{eq:ker}.
\begin{corollary}\label{res.solution}
The sequence $\{x_k\}$ is well defined, is contained in $B(x_0,t_ *),$ converges to a point $x_*\in B[x_0,t_ *]$  satisfying  $0\in f(x_*)+F(x_*).$ Moreover,   $x_k\in {K}(t_k)$, for $k=0,1\ldots$ and 
$$
\|x_*-x_k\|\leq t_ *-t_k, \qquad k=0,1\ldots.
$$
\end{corollary}
\begin{proof}
Since the mapping $x \mapsto L_f(x_0,  x )$ is strongly regular at $x_1$ for $0$, it follow from \eqref{eq:invplm} and Corollary~\ref{cor:ban} that
$
x_1=L_f (x_0,0)^{-1}\cap B(x_1,r_{x_1})
$
and the first Newton iterate is well defined.  Thus, from  ${\bf h1}$,  \eqref{eq.ipoint}  and definitions \eqref{eq:ker}  and \eqref{eq:kt}   we have
\begin{equation} \label{eq:fs}
\{x_0\}={K}(0)\subset {K}.
\end{equation}
We know from  Proposition~\ref{pr:syn}   that $N_{f+F}({K}) \subset {K}$. Thus,  using \eqref{eq:fs} and \eqref{eq.seq}  we conclude that the sequence $\{x_k\}$ is well defined and rests in ${K}.$  From the first inclusion on second part of the Proposition~\ref{pr:syn} we have trivially that $\{x_k\}$ is contained in $B(x_0,t_ *).$  To prove the convergence,  first we are going to prove by induction  that 
\begin{equation}\label{eq.defseq}
x_k\in {K}(t_k), \qquad k=0,1\ldots.
\end{equation}
The above inclusion, for $k=0$,  follows from \eqref{eq:fs}. Assume now that $x_k\in {K}(t_k).$ Then combining  Proposition~\ref{pr:syn}, \eqref{eq.seq} and \eqref{eq.majorfunc} we conclude that $x_{k+1}\in {K}(t_{k+1}),$ which completes the induction proof. Now,  using \eqref{eq.defseq} and \eqref{eq:ker} we have 
$$
\|L_f(x_k,0)^{-1}\cap B(x_1,r_{x_1})-x_k\| \leq -\frac{\psi(t_k)}{\psi'(t_k)}, \qquad k=0,1 \ldots,
$$
which, combined with  \eqref{eq.seq} and  definitions \eqref{NF} and \eqref{eq:DNS}  becomes 
\begin{equation}\label{des.conver}
\|x_{k+1}-x_k\|\leq t_{k+1}-t_k, \qquad k=0,1 \ldots.
\end{equation}
Taking into account that  $\{t_k\}$ converges to $t_ *,$   we  easily  conclude from  the above inequality  that
$$
\sum_{k=k_0}^{\infty} \|x_{k+1}-x_k\| \leq \sum_{k=k_0}^{\infty} t_{k+1}-t_k =t_ *- t_{k_0} < +\infty,
$$
for any $k_0 \in \mathbb{N}.$ Hence, we conclude that $\{x_k\}$ is a Cauchy sequence in $B(x_0, t_ *)$ and thus it  converges to some $x_* \in B[x_0, t_ *].$  Therefore, using  again \eqref{des.conver} we also conclude  that the inequality in the corollary holds.  

Now, we are going to show that $x_*$ is a solution to the generalized equation $f(x) + F(x) \ni 0.$ From inclusion in \eqref{eq:DNSa} we conclude
$$
\left(x_{k+1}, -f(x_k)-f'(x_k)(x_{k+1}-x_k)\right) \in \mbox{gph} ~F, \qquad k=0,1, \ldots. 
$$
Since $f$  is  continuous with continuous derivative  $f'$  in $\Omega$, $B[x_0, t_ *]\subset \Omega$ and $F$ has closed graph, last inclusion   implies that 
$$
(x_*, -f(x_*))=\lim_{k\to \infty }\left((x_{k+1}, -f(x_k)-f'(x_k)(x_{k+1}-x_k)\right) \in  \mbox{gph} ~F, 
$$
 which implies $f(x_*) + F(x_*) \ni 0$ and proof is complete.
\end{proof}
We have already proved that   the sequence $\{x_k\}$ converges to a solution $x_*$ of generalized equation $f(x)+F(x)\ni 0$ and $x_*\in B[x_0,t_*]$. Now, we will prove that $\{x_k\}$ converges $Q$-linearly and that $x^*$ is the unique solution of $f(x)+F(x)\ni 0$ in $B[x_0, t_ *]\cap B[x_{1}, r_{x_1}]$. Furthermore, by assuming that $\psi$ satisfies ${\bf h4}$, we will also prove that $\{x_k\}$ converges $Q$-quadratically.  For that, we need of the following result: 
\begin{lemma}\label{ine.rates}
Take $x,y\in B(x_0,R)$ and $0\leq \psi(0)\leq t <R$. If
\begin{equation} \label{eq:siqc}
t < t^*,\quad \|x-x_0\|\leq t, \quad \|y-x_1\|\leq r_{x_1},\quad \|y-x\|\leq t_*-t,  \quad 0\in f(y)+F(y), 
\end{equation}
then the following inequality holds
$$
\|y-N_{f+F}(x)\|\leq [t_*-n_{\psi}(t)]\frac{\|y-x\|^2}{(t_*-t)^2}.
$$
\end{lemma}
\begin{proof}
Since $0\in f(y) + F(y)$, using \eqref{eq:def.er} and \eqref{eq:pplm}, after some simple manipulations we obtain that
\begin{align*}
0\in f(y) + F(y)=  E_f(x,y) + L_f(x,y), 
\end{align*}
which by \eqref{eq:invplm} implies that  $y \in L_f(x,-E_f(x,y))^{-1}$.   Now,  we are going to prove that 
\begin{equation}  \label{eq:ebbro2}
E_{f}(x, y) \in B(0,r_0).
\end{equation}
Applying Lemma~\ref{pr:taylor} with $v=t_*$,   and using that $0\leq \psi(0)\leq t <t_*$ we have
$$
\lambda \|E_f(x,y)\| \leq \frac{1}{2}\psi''(t_*)(t_*-t)^2\leq \frac{1}{2}\psi''(t_*)(t_*-\psi(0))^2.
$$
On the other hand, Lemma~\ref{eq.ratemajor} give us  $t_*- n_{\psi}(0)\leq t_*/2$, which implies that $t_*- n_{\psi}(0)\leq n_{\psi}(0)=\psi(0)$. Therefore, above equation becomes 
$$
\lambda \|E_f(x,y)\| \leq \frac{1}{2}\psi''(t_*)\psi(0)^2,
$$
which under assumption  in \eqref{eq:conddef}  gives the desired inclusion in \eqref{eq:ebbro2}.  Since Proposition~\ref{wdns} implies that for any $x\in B(x_0,t^*)$, the mapping $ z\mapsto L_f(x,z)^{-1}\cap B(x_1, r_{x_1})$ is single-valued from $B(0, r_{0})$ to $B(x_1, r_{x_1})$. Thus,  taking into account third inequality in \eqref{eq:siqc}, inclusion in  \eqref{eq:ebbro2} and that   $y \in L_f(x,-E_f(x,y))^{-1}$,  we conclude that $y=L_f(x, -E_f(x,y))^{-1}\cap B(x_1,r_{x_1})$. Therefore,  combining  \eqref{NF} with  second part of Proposition~\ref{wdns}   we conclude
$$
\|y-N_{f+F}(x)\|= \|L_f(x, -E_f(x,y))^{-1}\cap B(x_1,r_{x_1})-L_f(x,0)^{-1}\cap B(x_1,r_{x_1})\| \leq -\frac{\lambda}{\psi'(t)}\|E_f(x,y)\|, 
$$
and since $t < t^*$, $\|x-x_0\|\leq t$ and $\|y-x\|\leq t_*-t$, we can apply  Lemma~\ref{pr:taylor} with $v=t_*$  to obtain 

$$
\|y-N_{f+F}(x)\|\leq -\frac{e_{\psi}(t, t_*)}{\psi'(t)} \frac{\|y-x\|^2}{(t_*-t)^2}. 
$$
But, due to $0\leq t<t_ *$ and  $\psi'(t)<0$,  using \eqref{eq:def.erf},  \eqref{eq.majorfunc} and  $\psi (t_*)= 0$ we have
$$
-\frac{e_{\psi}(t, t_*)}{\psi'(t)}=t_*-t + \frac{\psi(t)}{\psi'(t)}-\frac{\psi(t_*)}{\psi'(t)} = t_*-t + \frac{\psi(t)}{\psi'(t)}=t_*-n_{\psi}(t), 
$$
which  combined  with last inequality  gives the desired result. 
\end{proof}
\begin{corollary}
The sequences  $\{x_k\}$ and $\{t_k\}$ satisfy  the following inequality
\begin{equation}\label{ine.quadr}
\|x_*-x_{k+1}\|\leq \frac{t_*-t_{k+1}}{(t_* -t_k)^2}\|x_*-x_k\|^2, \qquad k=0,1\ldots.
\end{equation}
As a consequence,   the  sequence   $\{x_k\}$ converges  $Q$-linearly to the solution  $x^*$ as follows
\begin{equation}\label{ine.quadr1}
\|x_*-x_{k+1}\|\leq \frac{1}{2} \|x_*-x_k\|, \qquad k=0,1\ldots.
\end{equation}
Additionally, if $\psi$ satisfies ${\bf h4}$ then     the  sequence   $\{x_k\}$ converges $Q$-quadratically to $x_*$  as follows
\begin{equation}\label{ine.quadr2}
\|x_*-x_{k+1}\| \leq \frac{\psi''(t_*)}{-2\psi'(t_*)}\|x_*-x_k\|^2, \qquad k=0,1\ldots.
\end{equation}
\end{corollary}
\begin{proof}
We know, from Corollary~\ref{res.solution},  that $\{x_k\}$ is well defined, converges to $x_*$,  $\|x_{k}-x_0\|\leq t_k$  and  $\|x_*-x_k\|\leq t_ *-t_k$, for $k=0,1\ldots$.  Since  $\{x_k\}$ is well defined,  it follows from  \eqref{eq:DNS} that $x_k \in    B(x_{1}, r_{x_1})$ for $k=1, 2, \ldots$. Hence   $x_*\in  B[x_{1}, r_{x_1}]$, i.e., $\|x_*-x_1\|\leq r_{x_1}$.  Hence,  since {\bf h1} implies $t_1=n_{\psi}(0)=\psi(0)$  and $\{t_k\}$ is strictly increasing,  we can apply  Lemma~\ref{ine.rates} with $x=x_k,$ $y=x_*$ and $t=t_k$  to obtain
$$
\|x_*-N_{f+F}(x_k)\|\leq [t_*-n_{\psi}(t_k)]\frac{\|x_*-x_k\|^2}{(t_*-t_k)^2}.
$$
Thus  inequality \eqref{ine.quadr} follows from the above inequality, \eqref{eq.seq} and \eqref{eq.majseq}.  By the first part in Lemma~\ref{eq.ratemajor}, \eqref{eq.majseq} and  Corollary~\ref{res.solution} we have 
$$
\frac{t_*-t_{k+1}}{t_*-t_k}\leq \frac{1}{2}, \qquad  \qquad  \frac{\|x_*-x_k\|}{t_*-t_k}\leq 1.
$$
Combining these inequalities with \eqref{ine.quadr} we obtain \eqref{ine.quadr1}.
Now, assume that ${\bf h4}$ holds. Then, by Corollary~\ref{major.convergence}, the second inequality on \eqref{ine.rates1} holds, which combined with \eqref{ine.quadr} imply \eqref{ine.quadr2}. 
\end{proof}
\begin{corollary}
The limit $x_*$ of the sequence $\{x_k\}$ is the unique solution of the generalized equation $f(x)+F(x)\ni 0$ in  $B[x_0, t_ *]\cap B[x_{1}, r_{x_1}]$.
\end{corollary}
\begin{proof}
Corollary~\ref{res.solution} implies   that $\{x_k\}$ is well defined and   $\{x_k\}$ is contained in $B(x_0,t_ *)$, thus  it follows from  \eqref{eq:DNS} that $x_k \in   B(x_0, t_ *)\cap B(x_{1}, r_{x_1})$ for $k=1, 2, \ldots$. Hence   $x_*\in B[x_0, t_ *]\cap B[x_{1}, r_{x_1}]$.  Suppose there exist $y_* \in B[x_0, t_ *]\cap B[x_{1}, r_{x_1}]$ such that $y_*$ is  solution of $f(x)+F(x)\ni 0$. We will prove by induction that
\begin{equation}\label{iq.indu}
\|y_*-x_k\|\leq t_*-t_k, \qquad  k=0,1,\ldots.
\end{equation}
The case $k=0$ is trivial, because $t_0=0$ and $y_* \in  B[x_0,t_*]$. We assume that the inequality holds for some $k$.  First note that Corollary~\ref{res.solution} implies that  $x_k\in {K}(t_k)$, for $k=0,1\ldots$. Thus, from definition of ${K}(t_k)$  we conclude that   $\|x_k-x_0\|\leq t_k$, for $k=0,1\ldots$. Since   {\bf h1} implies $t_1=n_{\psi}(0)=\psi(0)$,  $\{t_k\}$ is strictly increasing  and $\|x_k-x_0\|\leq t_k$, we may apply Lemma~\ref{ine.rates} with $x=x_k$, $y=y_*$ and  $t=t_k$  to obtain
$$
\|y_*-N_{f+F}(x_k)\|\leq [t_*-n_{\psi}(t_k)]\frac{\|y_*-x_k\|^2}{(t_*-t_k)^2}, \qquad k=1, 2, \ldots.
$$ 
Using inductive hypothesis, \eqref{eq.seq} and \eqref{eq.majseq} we obtain, from latter inequality, that \eqref{iq.indu} holds for $k+1$. Since $x_k$ converges to $x_*$ and $t_k$ converges to $t_*$, from \eqref{iq.indu} we conclude that $y_*=x_*$. Therefore, $x_*$ is the unique solution of $f(x)+F(x)\ni 0$ in $B[x_0, t_ *]\cap B[x_{1}, r_{x_1}]$.
\end{proof}
\section{Special cases} \label{sec:scinmer}
In this section, we will present some special cases of  Theorem \ref{th:nt}.  It is worth pointing out that to find  a majorizing function for a given nonlinear function  is a very difficult problem and this is not our aim in this moment. On the other hand, there exist some classes of well known functions which  a majorant function is available, below we will present two examples, namely,   the classes of  functions  satisfying  a  Lipschitz-like and  Smale's  conditions, respectively.    In this sense,  the results obtained in  Theorem \ref{th:nt}  unify the convergence  analysis of Newton's method for the classes of generalized equations involving these functions, for instance,   Theorem~2 of  \cite{josephy1979} due to  N.~H.~Josephy and,  a particular instance of  Theorem~2 of \cite{Dontchev1996} due to A.~L.~Dontchev and a version of   Smale's theorem on Newton's method for analytical functions, see   \cite{BlumSmale1998}.

\subsection{ Kantorovich's theorem for Newton's method under Lipschitz condition}
In this section,  we will present a version of the  classical Kantorovich's theorem for Newton's method under Lipschitz-type condition  for generalized equations. The classical version for  $F\equiv \{0\}$  due to L.~V. Kantorovich have appeared, for example,  in  \cite{Kantorovich1948}, see also \cite{KantorovichAkilov1977} and for a historical perspective, see \cite{Polyak2004}.
\begin{theorem}\label{th:kngerl}
Let $\banacha$, $\banachb$ be Banach spaces, $\Omega\subseteq \banacha$ an open set,   $f: \Omega \to \banachb$ be continuous with   Fr\'echet derivative  $f'$ continuous and $F:\banacha \rightrightarrows  \banachb$ be a set-valued mapping with closed graph.  Suppose that the  partial linearization mapping $ L_f(x_0, . ):\banacha \rightrightarrows  \banachb$   at   $x_0$  is strongly regular at $x_1\in \Omega$ for $0$  with  associated  Lipschitz constant $\lambda >0$,  and there exists a constant $K>0$ such that  $B(x_0,1/K) \subset \Omega$  and  
\begin{equation*} \label{eq:hc}
 \lambda \|f'(y)-f'(x)\| \leq K\|y-x\|,\qquad \forall~x,\, y\in B(x_0,1/K).
\end{equation*}
Moreover, suppose that there exists   $b>0$ such that $bK\leq 1/2$ and  
$$
\|x_1-x_0\|\leq b. 
$$
Additionally, suppose that for $r_{0}$ and $r_{x_0}$    fixed  in  \eqref{eq:srcm}   the following  inequalities  hold:
$$
t_*=\frac{1-\sqrt{1-2bK}}{K}\leq r_{x_0},\qquad \qquad \frac{K}{2\lambda} b^2< r_0 .
$$
 Then, the sequence $\{x_k\}$ generated by Newton's method for solving the generalized equation $0\in f(x)+F(x)$ with starting point $x_0$ defined by
$$
 x_{k+1}:=L_{f}(x_k, 0)^{-1} \cap B(x_1, r_{x_1}), \qquad k=0,1,\ldots\,, 
$$
is well defined, $\{x_k\}$ is contained in $B(x_0, t_ *)$ and  converges to the point $x_*\in B[x_0, t_ *]$ which is the unique solution of $f(x)+F(x)\ni 0$ in $B[x_0, t_ *]\cap B[x_{1}, r_{x_1}]$, where $ r_{x_1}$ is fixed  in  \eqref{eq:srcm}.  Moreover, $\{x_k\}$ converges $Q$-linearly as follows
  \[
 \|x_*-x_{k+1}\| \leq \frac{1}{2}\|x_* -x_k\|, \qquad k=0,1,\ldots.
  \]
Additionally, if   $bK < 1/2$ then  the sequence $\{x_k\}$  converges $Q$-quadratically  as follows
$$
 \|x_*-x_{k+1}\| \leq \frac{K}{2\sqrt{1-2bK}}\|x_*-x_k\|^2, \qquad   k=0,1,\ldots.
$$
\end{theorem}
\begin{proof} 
Since $\psi:[0,1/K)\to \mathbb{R},$  defined by $\psi(t):=(K/2)t^2-t+b,$ is a majorant function for $f$ at point $x_0$, the result follows by invoking Theorem~\ref{th:nt}, applied to this particular context.
\end{proof}
\begin{remark}
The  above theorem, up to some minor adjustments,  merges to classical version, namely,   $F\equiv \{0\}$. Indeed,  for $F\equiv \{0\}$, the constants  in Corollary~\ref{cor:ban}  are $r_{0}=r_{x_1}=+\infty$ and $r_{x_0}=t_*$.
\end{remark}
We are going to  study an important  instance of  the generalized equation \eqref{eq:ipi}, namely, the generalized equation associated to $F=N_C$,    the normal cone  of  a nonempty,  closed and  convex subset  $C\subset \banacha$, 
\begin{equation} \label{eq:ncge}
f(x)+N_{C}(x) \ni 0.
\end{equation}
The next result is a version of classical convergence theorem for Newton's method under Lipschitz-type condition for the generalized equation  \eqref{eq:ncge}, it has been prove by N.~H.~Josephy in \cite{josephy1979}. 
\begin{theorem} \label{th:jtnm}
Let $\banacha$, $\banachb$ be Banach spaces, $C$ a nonempty,  closed and  convex subset of $\banacha$,  $\Omega\subseteq \banacha$ an open set and  $f:{\Omega}\to \banachb$ be continuous with Fr\'echet derivative $f'$ continuous  such that  
$$
 \|f'(x)-f'(y)\| \leq L \|x-y\|,\qquad x,\, y\in \Omega,
$$
where $L>0$. 
Moreover, suppose that $f(x_0)+f'(x_0)(x-x_0)+ N_{C}(x)$ is strongly regular at $x_1$ for $0$ with associated Lipschitz constant $\lambda >0$, $B(x_0,1/(\lambda K)) \subset \Omega$, there exists $b>0$ such that $b\lambda L\leq 1/2$ and  
\begin{equation*}
\|x_1-x_0\|\leq b.
\end{equation*}
Additionally, suppose that for $r_{0}$ and $r_{x_0}$ fixed  in \eqref{eq:srcm} the conditions   $t_*\leq r_{x_0}$ and $ Lb^2/{2}< r_0$ hold, where $t_ *=(1-\sqrt{1-2b\lambda L})/\lambda L$. Then, the sequence  generated by Newton's method, for solving $0\in f(x)+N_C(x)$, with starting point $x_0$, 
$$
 x_{k+1}:=L_{f}(x_k, 0)^{-1} \cap B(x_1, r_{x_1}), \qquad k=0,1,\ldots\,, 
$$
is well defined, $\{x_k\}$ is contained in $B(x_0, t_ *)$ and converges to the point $x_*\in B[x_0, t_ *]$ which is the unique solution of $0\in f(x)+N_C(x)$ in $B[x_0, t_ *]\cap B[x_1, r_{x_1}]$, where $ r_{x_1}$ is fixed  in  \eqref{eq:srcm}.  Moreover, $\{x_k\}$ converges $Q$-linearly as follows
  \[
 \|x_*-x_{k+1}\| \leq \frac{1}{2}\|x_* -x_k\|, \qquad k=0,1,\ldots.
  \]
Additionally, if    $b\lambda L< 1/2$ then the sequence $\{x_k\}$  converges $Q$-quadratically  as follows
$$
 \|x_*-x_{k+1}\| \leq \frac{\lambda L}{2\sqrt{1-2b\lambda L}}\|x_*-x_k\|^2, \qquad k=0,1, \ldots.
$$
\end{theorem}
\begin{proof}
Since $\psi:[0,1/K)\to \mathbb{R},$  defined by $\psi(t):=(\lambda L/2)t^2-t+b,$ is a majorant function for $f$ at point $x_0$, the result follows by invoking Theorem~\ref{th:kngerl} with $F=N_{C}$.
\end{proof}
\begin{remark}
The  above result  contain,  as particular instance,  several   theorem on Newton's method; see,  for example,  \cite{Daniel1973,  Kantorovich1948}.
\end{remark}
A. L. Dontchev \cite{Dontchev1996} under Aubin continuity of the mapping $L_f(x_0, \cdot )^{-1}:\mathbb{R}^{m} \rightrightarrows  \mathbb{R}^{n}$, defined by
\begin{equation} \label{eq:invplmnc}
L_{f+N_C}(x_0, z )^{-1}:=\left\{y\in \mathbb{R}^{n} ~:~ z\in f(x_0)+f'(x_0)(y-x_0)+N_C(y)\right\}, 
\end{equation}
 has shown that the Newton's method for solving \eqref{eq:ncge} generates a sequence that  converges
$Q$-quadratically to a solution. Now,  our purpose  is to show that, if $\banacha= \mathbb{R}^{m}$, $\ \banachb=\mathbb{R}^{n}$ , $F=N_C$ and $C \subset \mathbb{R}^n$ is a nonempty and  polyhedral convex set, then in this particular instance, Theorem~2 of \cite{Dontchev1996} follows from Theorem~\ref{th:kngerl}. We begin with the formal definition of Aubin continuity;  for more details see  \cite{DontchevRockafellar96, DontchevRockafellar2009}. First we need the following definitions:  The {\it distance} from a point $v\in \mathbb{R}^{n}$ to a set $U\subset \mathbb{R}^{n}$  is $d(v,  U):=\inf \{\|v-u\|~: u\in U\}$ and  the {\it excess} from   the set $U$ to the set  $V$ is $e(V, U):=\sup \{ d(v,  U) ~:~v\in V\}$.
\begin{definition}
A mapping  $H: \mathbb{R}^m \rightrightarrows \mathbb{R}^n$  is said to be  Aubin continuous, at $\bar{y}\in \mathbb{R}^m$ for $\bar{x}\in \mathbb{R}^n$,  if $\bar{x}\in H(\bar{y})$ and there exist  constants  $\alpha\geq 0$, $a>0$ and $c>0$ such that 
$$
e(H(y_1)\cap  B(\bar{x}, a),  H(y_2))\leq \alpha \|y_1-y_2\|, \qquad \forall ~ y_1, y_2 \in B(\bar{y}, c).
$$
\end{definition}
It has been shown in  \cite[Theorem 1]{DontchevRockafellar96} that if $C\subset \mathbb{R}^n$ is  a polyhedral convex set, then   Aubin continuity of $L_{f+N_C}(x_0, \cdot )^{-1}$  is equivalent to strong regularity of $f+N_{C}$. Next we state, with  some  adjustments,  Theorem~2 of \cite{Dontchev1996}. 

\begin{theorem}
Let  $C\subset \mathbb{R}^n$ be a polyhedral convex set,  $\Omega\subseteq \mathbb{R}^n$ an open set and  $f:{\Omega}\to \banachb$ be continuous with  derivative  $f'$ continuous  such that  
$$
 \|f'(x)-f'(y)\| \leq L \|x-y\|,\qquad  \forall ~ x,\, y\in \Omega, 
$$
where $L>0$. Let $x_0\in \Omega$ and suppose that $\|x_1-x_0\|\leq b$, $L_{f+N_C}(x_0, \cdot )^{-1}: \mathbb{R}^m \rightrightarrows \mathbb{R}^n$ defined in \eqref{eq:invplmnc} is Aubin continuous at $0\in \mathbb{R}^m$ for $x_1 \in \mathbb{R}^n$ with modulus $\alpha\geq 0$ and associated constantes  $a>0$ and $c>0$,  $B(x_0,1/(\alpha L)) \subset \Omega$  and $\alpha bL\leq 1/2$. Additionally, suppose that for $r_{0}$ and $r_{x_0}$ fixed in \eqref{eq:srcm} the conditions $t_*\leq  \min \{a, r_{x_0}\}$ and  $ Lb^2/{2}< \min \{c, r_0\}$ hold, where $t_ *=(1-\sqrt{1-2\alpha bL})/\alpha L$. Then, the sequence  generated by Newton's method, for solving $0\in f(x)+N_C(x)$, with starting point $x_0$$$
 x_{k+1}:=L_{f}(x_k, 0)^{-1} \cap B(x_1, r_{x_1}), \qquad k=0,1,\ldots\,, 
$$
is  well defined,  $\{x_k\}$ is contained in $B(x_0, t_*)$ and converges to the point $x_*$ which is the unique solution of $f(x)+N_C(x)\ni 0$ in $B[x_0, t_ *]\cap B[x_1, r_{x_1}]$, where $ r_{x_1}$ is fixed  in  \eqref{eq:srcm}.   Moreover, $\{x_k\}$ converges $Q$-linearly as follows
  \[
 \|x_*-x_{k+1}\| \leq \frac{1}{2}\|x_* -x_k\|, \qquad k=0,1,\ldots.
  \]
Additionally, if    $\alpha bL<1/2$ then  the sequence $\{x_k\}$  converges $Q$-quadratically  as follows  
$$
 \|x_*-x_{k+1}\| \leq \frac{\alpha L}{2\sqrt{1-2\alpha Lb}}\|x_*-x_k\|^2, \qquad k=0,1, \ldots.
$$
\end{theorem}
\begin{proof}
Since $C\subset \mathbb{R}^n$ is a polyhedral convex set, \cite[Theorem 1]{DontchevRockafellar96} implies that Aubin continuity of $L_{f+N_C}(x_0, \cdot )^{-1}$ at $0\in \mathbb{R}^m$ for $x_1 \in \mathbb{R}^n$ with modulus $\alpha\geq 0$, is equivalent to strongly regularity of $f(x_0)+f'(x_0)(x-x_0)+N_{C}(x)$ at $0$ for $x_1$ with associated Lipschitz constant $\alpha\geq 0$. Thus, the result  follows by  applying   Theorem~\ref{th:jtnm}. 
\end{proof}
\subsection{Smale-type theorem for Newton's method}
In this section, we will present a version of classical convergence theorem for Newton's method under Smale-type condition for generalized equations, for example,   see  \cite{BlumSmale1998}.   
\begin{theorem} \label{th:kngesrs}
Let $\banacha$, $\banachb$ be Banach spaces, $\Omega\subseteq \banacha$ an open set and $f:{\Omega}\to \banachb$ be an analytic mapping, $F:\banacha \rightrightarrows  \banachb$ be a set-valued mapping with closed graph and $x_0 \in \Omega$.  Suppose that the  partial linearization mapping $ L_f(x_0,  . ):\banacha \rightrightarrows  \banachb$   at   $x_0$,  is strongly regular at $x_1\in \Omega$ for $0$ with  associated  Lipschitz constant $\lambda >0$  and 
\begin{equation} \label{eq:SmaleCond}
   \gamma:= \sup _{ n > 1 }\left\| \frac  {\lambda f^{(n)}(\bar{x})}{n!}\right\|^{1/(n-1)}<+\infty.
\end{equation}
Moreover, suppose that $B(x_0, 1/\gamma)\subseteq \Omega$ and there exists $b>0$ such that $\|x_1-x_0\|\leq b$ and $b\gamma \leq 3-2\sqrt{2}$.  Additionally, suppose that for $r_{0}$ and $r_{x_0}$ fixed in \eqref{eq:srcm} the conditions 
\begin{equation} \label{eq:bsc} 
t_*\leq r_{x_0}, \qquad \qquad \qquad \frac{4^3\gamma b^2}{\lambda \left(3-b\gamma+\sqrt{(b\gamma+1)^2-8b\gamma}\right)^3}<r_0,
\end{equation} 
hold, where $t_ *=(b\gamma +1-\sqrt{(b\gamma+1)^2 -8b\gamma})/4\gamma$. Then, the sequence generated by Newton's method for solving $f(x)+F(x)\ni 0$ with starting point $x_0$, 
$$
 x_{k+1}:=L_{f}(x_k, 0)^{-1} \cap B(x_1, r_{x_1}), \qquad k=0,1,\ldots\,, 
$$
is well defined, $\{x_k\}$ is contained in $B(x_0,t_ *)$ and converges to the point $x_*$,  which is the unique solution of $f(x)+F(x)\ni 0$ in $B[x_0, t_ *]\cap B[x_1, r_{x_1}]$, where $ r_{x_1}$ is fixed  in  \eqref{eq:srcm}. Moreover, $\{x_k\}$ converges $Q$-linearly as follows
  \[
 \|x_*-x_{k+1}\| \leq \frac{1}{2}\|x_* -x_k\|, \qquad k=0,1,\ldots.
  \]
Additionally, if $b\gamma < 3-2\sqrt{2}$, then $\{x_k\}$ converges $Q$-quadratically as follows
$$
	\|x_*-x_{k+1}\| \leq \frac{\gamma}{(1-\gamma t_ *)[2(1-\gamma t_ *)^2-1]}\|x_*-x_k\|^2,\qquad k=0,1,\ldots.
$$
\end{theorem}
Before proving above theorem we need of two results. The next results gives a condition  that is easier to check than condition
\eqref{Hyp:MH}, whenever the mapping under consideration are twice continuously differentiable, and its proof follows the same path of Lemma~21 of \cite{FerreiraMax2013}. 
\begin{lemma}\label{lem.cond1}
Let $\Omega \subset \banacha$ be an open set, and let $f:{\Omega}\to \banachb$ be an analytic function. Suppose that $x_0 \in \Omega$ and $B(x_0, 1/ \gamma)\subset \Omega,$ where $\gamma$ is defined in \eqref{eq:SmaleCond}. Then for all $x\in B(x_0, 1/  \gamma),$ it holds that
$
\|f''(x)\|\leq 2  \gamma/(1-  \gamma\|x-x_0\|)^3.
$
\end{lemma}
The next result gives a relationship between the second derivatives $f''$ and $ \psi''$, which allow us to show that $f$ and $\psi$ satisfy \eqref{Hyp:MH}, and its proof is similar to Lemma~22 of \cite{FerreiraMax2013}. 
\begin{lemma} \label{lc}
Let $\banacha$, $\banachb$ be Banach spaces, $\Omega\subseteq \banacha$ be an open set, $f:{\Omega}\to \banachb$ be twice continuously differentiable. Let $x_0 \in \Omega$, $R>0$ and $\kappa=\sup\{t\in [0, R): B(x_0, t)\subset \Omega\}$. Let $\lambda >0$ and \mbox{$\psi:[0,R)\to \mathbb {R}$} be twice continuously differentiable such that $\lambda \|f''(x)\|\leqslant \psi''(\|x-x_0\|),$
for all $x\in B(x_0, \kappa)$, then $f$ and $\psi$ satisfy \eqref{Hyp:MH}.
\end{lemma}

\noindent
{\bf [Proof of Theorem \ref{th:kngesrs}]}.
Consider $\psi:[0, 1/ \gamma) \to \mathbb{R}$ defined by $\psi(t)=t/(1- \gamma t)-2t+b$. Note that $\psi$ is  analytic and 
$\psi(0)=b>0$, $\psi'(t)=1/(1- \gamma t)^2-2$, $\psi'(0)=-1$, $\psi''(t)=2 \gamma/(1-\gamma t)^3$ and $\psi(t_*)=0$. 
It follows from the last  equalities  that $\psi$ satisfies {\bf h1},  {\bf h2},   {\bf h3} and \eqref{eq:bsc} .  Combining  Lemma~\ref{lc}  with  Lemma~\ref{lem.cond1}, we conclude  that $f$  and $\psi$ satisfy  \eqref{Hyp:MH}. Therefore, the result follows by applying the Theorem~\ref{th:nt}.
\qed

\section{Final remarks} \label{sec:fr}
In this paper we have obtained  a semi local convergence result to Newton's method for solving generalized equation in Banach spaces and  under the majorant condition. As future works, we propose to study this method using the approach of this paper under a weak assumption than strong regularity, namely, the regularity metric; see \cite{DontchevRockafellar2009}. It is well  known  that the inexact analysis support the efficient computational implementations of the exact ones and,  as we have seen above, the majorant condition  allowed us  to unify several convergence results pertaining to Newton's method.   So, unifying result for  inexact versions of Newton's method would be very  welcome.

\end{document}